\documentclass{amsart}
\usepackage{graphicx} 
\usepackage{comment}
\usepackage[margin=25mm]{geometry}
\usepackage{amsmath}
\usepackage{amsfonts}
\usepackage{amssymb}
\usepackage{amsthm}
\usepackage{mathtools} 
\usepackage{enumitem} 
\usepackage{hyperref}
\hypersetup{
    colorlinks,
    linkcolor=black,
    citecolor=black,
    filecolor=black,      
    urlcolor=black,
}

\newtheorem{theorem}[subsection]{Theorem}
\newtheorem{lemma}[subsection]{Lemma}
\newtheorem{prop}[subsection]{Proposition}
\newtheorem{corollary}[subsection]{Corollary}

\theoremstyle{definition}
\newtheorem{definition}[subsection]{Definition}
\newtheorem{notation}[subsection]{Notation}

\newtheorem{remark}[subsection]{Remark}

\newcommand{\fsl}{\mathfrak{sl}}
\newcommand{\fgl}{\mathfrak{gl}}
\newcommand{\fg}{\mathfrak{g}}
\newcommand{\fh}{\mathfrak{h}}
\newcommand{\bz}{\mathbb{Z}}
\newcommand{\bc}{\mathbb{C}}
\DeclareMathOperator{\spann}{span}

\title{Finite-dimensional irreducible representations of Yangians}
\author{Mikhail Gradov and Aleksei Ilin}
\date{}

\begin{document}

\begin{abstract}
We present a direct proof of the well-known classification of finite-dimensional representations of the Yangian $Y(\fsl_2)$ following a direct approach presented in \cite{Pressley1991} for $U_q(\mathfrak{ \hat sl_2})$. 

\end{abstract}

\maketitle

\section{Introduction}
\label{section: Introduction}
Yangian for a simple Lie algebra $\fg$ is a unique graded deformation of a Hopf algebra $U(\fg[t])$ (such that $\fg[t]$ carries a standard Lie bialgebra structure); see \cite{Drinfeld1985}, \cite{gwx}. 
Finite-dimensional irreducible representations of the Yangian are classified in \cite{Drinfeld1987} with a link to the V. Tarasov works \cite{Tarasov1984}, \cite{Tarasov1985} and the claim that the general case can be deduced from the $Y(\fsl_2)$ case. However, the papers of V. Tarasov are written using the language of the quantum inverse scattering method and representations of $Y(\fgl_2)$ are classified.  The classification of finite-dimensional irreducible representations of $Y(\fgl_2)$ which uses Yangian's terminology 
was given by Molev \cite{Molev1998}. In order to get classification for $Y(\fsl_2)$ a (non-trivial) decomposition 
 $
 Y(\mathfrak{gl}_{2})\cong Z(Y(\fgl_2)) \otimes Y(\fsl_{2}) 
 $ is used (here $Z(Y(\fgl_2))$ is the center of an algebra $Y(\fgl_2)$). We also note that the proof of parameterisation in this case contains explicit construction of representations.
However, from the Lie theory point of view, it is natural first to classify representations of $Y(\fsl_2)$. When we classify irreducible representations of the Lie algebra $\fsl_2$, we obtain a necessary condition just from one relation.  

In the case of the quantum affine algebra $U_q(\hat \fsl_2)$, the classification in this way was done by V. Chari and A. Presley \cite{Pressley1991}.
The main result of the present paper is a direct proof of the parameterization of finite-dimensional irreducible representations of $Y(\fsl_2)$. As in the $\fsl_2$-case, the necessary condition for the parameters follows from one relation in $Y(\fsl_2)$; see Proposition \ref{prop: 2.8}.
For completeness, we also discuss how to prove the theorem for $Y(\fg)$ from the $Y(\fsl_2)$-case. 

\subsection*{Acknowledgments} We thank L. Rybnikov for useful remarks. This article is an output of a research project implemented as part of the Basic Research Program at the National Research University Higher School of Economics (HSE University).

\section{ Yangian \texorpdfstring{$Y(\fg)$}{Y(fg)} }
\label{section: Yangian}
Let $\fg$ be a complex simple Lie algebra, $\fh \subset \fg$ be a Cartan subalgebra, $\{\alpha_1, \ldots, \alpha_n\}$b e a set of simple roots,  $a_{ij} = \dfrac{2(\alpha_i,\alpha_j)}{(\alpha_i,\alpha_i)}$, $d_i = \dfrac{(\alpha_i,\alpha_i)}{2}$. Let
$Q=\bigoplus_{i} \mathbb{Z}\alpha_{i}$ be a root lattice and $Q_{+}=\bigoplus_{i} \mathbb{Z}_{\geq 0} \alpha_{i}$ be a dominant weights.

\subsection{Yangian for a simple Lie algebra \texorpdfstring{$\fg$}{g}.}
\label{subsection: Yangian_for_a_simple_Lie_algebra}
\begin{definition}
     \label{def:2.2}
     The Yangian $Y(\fg)$ is isomorphic to the associative algebra with generators $X_{i,r}^{\pm}$, $H_{i,r}$, $i=1,...,n$, $r \in \bz_{\geq 0}$, and the following defining relations:
     \begin{gather}
         [H_{i,r},H_{j,s}]=0, \tag{Y1} \label{Y1} \\
         [H_{i,0}, X_{j,s}^{\pm}]=\pm d_{i}a_{i,j}X_{j,s}^{\pm},  \tag{Y2} \label{Y2}\\
         [H_{i,r+1}, X_{j,s}^{\pm}]-[H_{i,r}, X_{j,s+1}^{\pm}]=\pm \frac{1}{2}d_{i}a_{ij}(H_{i,r}X_{j,s}^{\pm}+X_{j,s}^{\pm}H_{i,r}). \tag{Y3} \label{Y3} \\
         [X_{i,r}^{+}, X_{j,s}^{-}]=\delta_{i,j}H_{i,r+s}, \tag{Y4} \label{Y4} \\
         [X_{i,r+1}^{\pm}X_{j,s}^{\pm}]-[X_{i,r}^{\pm}, X_{j,s+1}^{\pm}]=\pm \frac{1}{2}d_{i}a_{ij}(X_{i,r}^{\pm}X_{j,s}^{\pm}+X_{j,s}^{\pm}X_{i,r}^{\pm}), \tag{Y5} \label{Y5} \\
         \sum_{\pi}[X_{i, r_{\pi_{1}}}^{\pm},[X_{i, r_{\pi_{2}}}^{\pm},...,[X_{i,r_{\pi_{m}}}^{\pm},X_{j,s}^{\pm}]...]]=0, \tag{Y6} \label{Y6}
     \end{gather}
     for all sequences of non-negative integers $r_{1},...,r_{m}$, where $m=1-a_{ij}$ and the sum is over all permutations $\pi$ of $\{ 1,...,m \}$. 
 \end{definition}

\begin{remark}
    For $Y(\fsl_2)$ we will use notation $e_i = X_{1,i}^+, f_i = X_{1,i}^-, h_i = H_{1,i}$.
\end{remark}
We also define $X_i^{\pm}(u) = \sum_{r \geq 0} X_{i,r}^{\pm} u^{-r}$ and $H_i(u) = 1 + \sum_{r \geq 1} H_{i,r-1} u^{-r}$.
Following Levendorskii \cite{Levendorskii1993}, we equip $Y(\fg)$ with the
filtration in which $\deg X^{\pm}_{i,r} = \deg H_{i,r} = r$, and denote by
$Y(\fg)_{r}$ the span of the monomials in the generators of degree at most $r$.

\begin{notation}
\label{not: root vectors}
For each $\alpha \in \Delta^{+}$ for a Lie algebra $\fg$ fix a decomposition
$\alpha = \alpha_{i_1} + \dots + \alpha_{i_p}$ into simple roots such that
\[
x^{\pm}_{\alpha}
= [x^{\pm}_{i_1}, [x^{\pm}_{i_2}, \dots, [x^{\pm}_{i_{p-1}}, x^{\pm}_{i_p}]
\cdots]]
\]
is a nonzero root vector of $\fg$ in the $\pm\alpha$ root space. Given a
decomposition $r = r_1 + \dots + r_p$ into nonnegative integers, put
\[
X^{\pm}_{\alpha,r}
= [X^{\pm}_{i_1,r_1}, [X^{\pm}_{i_2,r_2}, \dots,
[X^{\pm}_{i_{p-1},r_{p-1}}, X^{\pm}_{i_p,r_p}] \cdots]],
\qquad
H_{\alpha,r} = [X^{+}_{\alpha,0}, X^{-}_{\alpha,r}],
\]
and $H_{i,r} = H_{\alpha_i,r}$.
\end{notation}

\begin{remark}
\label{rem: independence of decomposition}
The above elements depend on the chosen decomposition of $r$, but only up to lower
filtration degree: if $\widetilde{X}^{\pm}_{\alpha,r}$ is obtained from another
decomposition of $r > 0$, then
$X^{\pm}_{\alpha,r} - \widetilde{X}^{\pm}_{\alpha,r} \in Y(\fg)_{r-1}$, and
likewise for $H_{\alpha,r}$.
\end{remark}

\begin{theorem}[\cite{Levendorskii1993}]
\label{thm: pbw}
Let $\prec$ be an arbitrary total ordering on the set
\[
\{X^{+}_{\alpha,r} \mid \alpha \in \Delta^{+},\ r \geq 0\}
\cup \{X^{-}_{\alpha,r} \mid \alpha \in \Delta^{+},\ r \geq 0\}
\cup \{H_{i,r} \mid 1 \leq i \leq n,\ r \geq 0\}.
\]
Then the ordered monomials form a basis of $Y(\fg)$.
\end{theorem}

\begin{corollary}
\label{cor: triangular}
$Y(\fg) = Y^{-}\,Y^{0}\,Y^{+}$.
\end{corollary}

\begin{proof}
Choose $\prec$ so that $X^{-}_{\alpha,r} \prec H_{i,s} \prec X^{+}_{\beta,t}$
for all $\alpha, \beta, i, r, s, t$. Every ordered monomial is then a product of
a monomial in the $X^{-}_{\alpha,r}$, a monomial in the $H_{i,r}$ and a monomial
in the $X^{+}_{\alpha,r}$, hence lies in $Y^{-}Y^{0}Y^{+}$. By
Theorem~\ref{thm: pbw} such monomials span $Y(\fg)$.
\end{proof}

\subsection{A relation in \texorpdfstring{$Y(\fsl_2)$}{Y(sl2)}}
\label{subsection: relation}


\begin{prop}
\label{prop: 2.8}
\[
    e_{1}^{k}f_{0}^{k+1}v=(k+1)f_{0}e_{1}^{k}f_{0}^{k}v+\frac{k(k+1)}{2}[h_{1}, e_{1}^{k-1}f_{0}^{k}]v+(k-1)k(k+1)e_{1}^{k-1}f_{0}^{k}v
    \]
\end{prop}

The proof is contained in Appendix A.

\subsection{Embedding of \texorpdfstring{$Y(\fsl_2)$}{Y(sl2)}}
\label{subsection: embedding}

\begin{prop}
\label{prop: 2.10}
Subalgebra $Y_i$ of $Y(\fg)$ generated by $H_{i,r}, X_{i,r}^{\pm}, r \geq 0$ is isomorphic to $Y(\fsl_2)$.

\end{prop}
\begin{proof}
Consider the map
$$\psi_{i}: Y(\fsl_2) \to Y(\fg), H_{i,r} \mapsto {d_i}^{r+1}H_{i,r}, X_{i,r}^{\pm} \mapsto {d_i}^{r+\frac{1}{2}}X_{i,r}^{\pm}$$
It is easy to check that this map preserves $Y(\fsl_2)$ relations, the image coincides with $Y_i$, and the kernel is zero by the PBW theorem.
\end{proof}

\section{Finite-dimensional representations of \texorpdfstring{$Y(\fg)$}{Y(g)}}

\subsection{Highest-weight representations.}
\label{subsection: highest_weight}
Let $\lambda_{i,r}$, for $i=1,...,n$, $r \in \mathbb{Z}_{\geq 1}$, be a collection of complex numbers, and let $\lambda_i(u) = 1 + \sum_{r \geq 1} \lambda_{i,r-1} u^{-r}$ and $\lambda(u) = (\lambda_1(u), \ldots, \lambda_n(u))$. Let $V$ be a representation of $Y(\fg)$ and we define {\em a weight space}
$$V_{\lambda(u)} = \{ v\in V | H_i(u) \cdot v = \lambda_i(u) v, \quad \forall \, i=1, \ldots, n\}. $$

We say that $v \in V$ is a primitive vector if $X_i^+(u) \cdot v = 0$ for any $i = 1, \ldots, n$ and $v \in V_{\lambda(u)}$.

 \begin{definition}
 \label{def: 3.2}
A representation $V$ of $Y(\fg)$ is called a highest weight representation if there exists a primitive vector $v \in V$ such that $V = Y(\fg) v$. We call such a vector a highest-weight vector.  
 \end{definition}

The following is well-known(see e.g. \cite{Chari1995}):
\begin{prop}
\label{prop:Highest_weight}
\begin{itemize}
    \item Every irreducible $Y(\fg)$-module $V$ is a highest-weight module.
    \item There exists a unique irreducible representation $L_{\lambda(u)}$ of the highest weight $\lambda(u) = (\lambda_1(u), \ldots, \lambda_n(u))$.
\end{itemize}
\end{prop}

\subsection{Main theorem}
\label{subsection: Main}
\begin{theorem}
 \label{thm: 3.5}
     The irreducible $Y(\fg)$-module $L_{\lambda(u)}$ is finite-dimensional if and only if there exist polynomials $P_{i}\in \mathbb{C}[u]$ such that 
     \[
    \lambda_i(u) =   \frac{P_{i}(u+d_{i})}{P_{i}(u)}
     \]
     in the sense that the right-hand side is the Laurent expansion of the left-hand side about $u=\infty$.
 \end{theorem}

This theorem is formulated in \cite{Drinfeld1987} as considered for $Y(\fsl_2)$ in \cite{Tarasov1984}, \cite{Tarasov1985} and that it can be deduced from $Y(\fsl_2)$-case for an arbitrary $Y(\fg)$.
We will give a direct proof for $Y(\fsl_2)$ which only uses defining relations, analogous to the case of $U_q(\hat \fsl_2)$ given in \cite{Pressley1991}. Then we discuss a proof for $Y(\fg)$, which is not written explicitly in the literature.\footnote{However, there is a relatively full proof in the lecture notes by S. Gautam \cite{gautam}. Note that in the case of $Y(\fsl_2)$ these lecture notes use Molev's approach.}

\subsection{Drinfeld polynomial for \texorpdfstring{$Y(\fsl_{2})$}{Y(sl2)}}
\label{subsection: Drinfeld}

\begin{prop}
\label{prop: 3.7}
For any $\lambda(u) \in \mathbb{C}[[u^{-1}]]$ such that $\lambda(u) = 1 + \lambda_0 u^{-1}+ \ldots$, $\lambda_0 = n  \in \mathbb{Z}_{\geq 0}$ there exists unique $P(u) = u^n +p_{n-1} u^{n-1} + \ldots + p_0 + p_{-1} u^{-1} + \ldots \in \mathbb{C}((u^{-1}))$ such that
$\lambda(u) = \dfrac{P(u+1)}{P(u)}$.
Moreover, 
\begin{equation}
p_{r}=\frac{1}{r-n}\sum_{t=r+1}^{n}p_{t} \left( \lambda_{t-r}- \binom{t}{r-1}
\right)
\label{(1)}
\end{equation}
$r=0,1,...,n-1$ and
$\binom{t}{-1} =0$
for all $t$. 

\begin{equation}
p_{-r}=\frac{1}{(r+n)}\left[\sum_{t=1}^{r-1}p_{-t} \left( (-1)^{r-t-1} \binom{r}{r+1-t}
-\lambda_{r-t} \right) -\sum_{t=0}^{n}p_{t}\lambda_{t+r}\right]
\label{(2)}
\end{equation} for $r \geq 1$.

\end{prop}
\begin{proof}
Consider an equation on $P(u)$:
\[
P(u+1)=P(u)\lambda(u)
\] 
Coefficientwise, it gives the following equations ($p_n=1)$:
\[
\sum_{t=r}^{n}p_{t} \binom{t}{r}
=p_{r}+\sum_{t=r+1}^{n}p_{t}\lambda_{t-r-1}
\]
\[
\sum_{t=1}^{r}p_{-t} \binom{r-1}{r-t}
(-1)^{r-t}=p_{-r}+p_{-r+1}\lambda_{0}+p_{-r+2}\lambda_{1}+...+p_{-1}\lambda_{r-2}+\sum_{t=0}^{n}p_{t}\lambda_{t+r-1}
\]

After a change of index, we obtain the result.
\end{proof}

\subsection{Proof of if part for \texorpdfstring{$Y(\fsl_{2})$}{Y(sl2)}.}
Suppose that $\dim L(\lambda(u)) < \infty$. In this case there exists $n \in \mathbb{Z}_{\geq 0}$ such that $\lambda_0 = n$, because the highest weight vector with respect to $Y(\fsl_2)$ is at the same time the highest weight vector with respect to $U(\fsl_2)$. Hence from Proposition \ref{prop: 3.7} we know that $\lambda(u) = \dfrac{P(u+1)}{P(u)}$ for some $P(u) = u^n +p_{n-1} u^{n-1} + \ldots + p_0 + p_{-1} u^{-1} + \ldots \in \mathbb{C}((u^{-1}))$. It is enough to show that $p_{-r} = 0$ for all $r \geq 0$.
\begin{lemma}
$$\sum_{i=0}^n p_i (f_i v) =0 $$ implies $p_{-r} = 0$ for any $r \geq 0$.
 \end{lemma}
\begin{proof}
The proof is by induction on $r$.
For $r = 1$ we have
$$0 = e_1\cdot \left(\sum_{i=0}^n p_i (f_i \cdot v)\right) =  \left( \sum_{t=0}^{n}p_{t}h_{t+1} \right) v = \left( \sum_{t=0}^{n}p_{t}\lambda_{t+1} \right) v = p_{-1} v$$
For arbitrary $r$, using the recurrence relation, we get
$$0 = e_r\cdot \left(\sum_{i=0}^n p_i (f_i \cdot v)\right) =  \left( \sum_{t=0}^{n}p_{t}h_{t+r} \right) v = \left( \sum_{t=0}^{n}p_{t}\lambda_{t+r} \right) v = p_{-r} v  $$
\end{proof}
\begin{lemma}
\label{prop: 4}
    \[
    e_{1}^{k}f_{0}^{k+1}v=\sum_{i=0}^{k}\alpha_{i,k}f_{i}v
    \] where
    $\alpha_{0,0} = 1$ and 
     \[
    \alpha_{0,k}=(k+1)\left[\sum_{i=0}^{k-1}\alpha_{i,k-1}\lambda_{i+1}+k(k-\lambda_{0})\alpha_{0,k-1} \right];
    \]
    \[
    \alpha_{i,k}=k(k+1)\left[(k-\lambda_{0})\alpha_{i,k-1}-\alpha_{i-1,k-1}\right];
    \]
    \[
    \alpha_{k,k}=-k(k+1)\alpha_{k-1,k-1}.
    \]
    
\end{lemma}
\begin{proof}
For $k=0$ we have $f_0 v = 1 \cdot f_0 v$.
    We prove the statement by induction on $k$. 
    By Proposition \ref{prop: 2.8} we have:
    $$e_{1}^{k}f_{0}^{k+1}v=(k+1)f_{0}e_{1}^{k}f_{0}^{k}v+\frac{k(k+1)}{2}[h_{1}, e_{1}^{k-1}f_{0}^{k}]v+(k-1)k(k+1)e_{1}^{k-1}f_{0}^{k}v$$
    Moreover
    
    \[
    \begin{split}
    f_0e_{1}^{k}f_{0}^{k}v=f_0e_{1} \Big( \sum_{i=0}^{k-1}\alpha_{i,k-1}f_{i}v \Big)=f_0\sum_{i=0}^{k-1}\alpha_{i,k-1}e_{1}f_{i}v=\Big(\sum_{i=0}^{k-1}\alpha_{i,k-1}\lambda_{i+1} \Big) f_0 v
    \end{split}
    \]

    \[
    \begin{split}
        [h_{1}, e_{1}^{k-1}f_{0}^{k}]v=h_{1}e_{1}^{k-1}f_{0}^{k}v-e_{1}^{k-1}f_{0}^{k}h_{1}v=h_{1}\sum_{i=0}^{k-1}\alpha_{i,k-1}f_{i}v-\lambda_{1}\sum_{i=0}^{k-1}\alpha_{i, k-1}f_{i}v=\\
        =\sum_{i=0}^{k-1}\alpha_{i,k-1}h_{1}f_{i}v-\lambda_{1}\sum_{i=0}^{k-1}\alpha_{i, k-1}f_{i}v=\sum_{i=0}^{k-1}\alpha_{i,k-1}[h_{1}, f_{i}]v
    \end{split}
    \]
    Using the commutation relation
    \[
    \begin{split}
        [h_{1}, f_{i}]-[h_{0}, f_{i+1}]=-h_{0}f_{i}-f_{i}h_{0} \Leftrightarrow [h_{1}, f_{i}]=-2f_{i+1}-2f_{i}h_{0}+2f_{i}
    \end{split}
    \]
we obtain
    \[
    \begin{split}
        [h_{1}, e_{1}^{k-1}f_{0}^{k}]v=\sum_{i=0}^{k-1}\alpha_{i,k-1}\Big( -2f_{i+1}-2f_{i}h_{0}+2f_{i} \Big)v=\\
        =-2\sum_{i=0}^{k-1}\alpha_{i,k-1}f_{i+1}v-2\lambda_{0}\sum_{i=0}^{k-1}\alpha_{i,k-1}f_{i}v+2\sum_{i=0}^{k-1}\alpha_{i,k}f_{i}v=\\
        =2(1-\lambda_{0})\sum_{i=0}^{k-1}\alpha_{i,k-1}f_{i}v-2\sum_{i=0}^{k-1}\alpha_{i,k-1}f_{i+1}v
    \end{split}
    \]
It implies
\[
\begin{split}
     e_{1}^{k}f_{0}^{k+1}v=\Big( (k+1) \sum_{i=0}^{k-1}\alpha_{i,k-1}\lambda_{i+1} + k(k+1)(1-\lambda_{0})\alpha_{0, k-1}+(k-1)k(k+1)\alpha_{0, k-1} \Big) f_{0}v+\\
    +\sum_{i=1}^{k-1}\Big( k(k+1)(1-\lambda_{0})\alpha_{i, k-1}-k(k+1)\alpha_{i-1, k-1}+(k-1)k(k+1)\alpha_{i, k-1} \Big) f_{i}v
    -k(k+1)\alpha_{k-1, k-1}f_{k}v = \sum_{i=0}^k \alpha_{ik} f_i v
\end{split}
\]

\end{proof}

Now let $\lambda_0=n$ and compute $\alpha_{i,n}, i =0, \ldots, n$.
\begin{prop}
\label{prop: 3.10}
    \[
    e_{1}^{n}f_{0}^{n+1}v= (-1)^n n! (n+1)! \sum_{i=0}^{n} p_i f_{i}v
    \]
\end{prop}
\begin{proof}
    Let $b_{i, k} = \dfrac{(-1)^k \alpha_{i,k}}{k!(k+1)!}$ and
    $Q_k(x) = \sum_{i=0}^k b_{i,k} x^k$. Note that $Q_0=1$.
    \begin{lemma}
    Recurrence relations on $\alpha_{i,k}$ are equivalent to  a recurrence relation
   \begin{equation} \label{qrec} Q_k(x) = (x+n-k)Q_{k-1}(x)-\dfrac{1}{k} l(Q_{k-1})\end{equation} where $l: \bc[x] \to \bc, \sum_{i=0}^m c_i x^i \mapsto \sum_{i=0}^m c_i \lambda_{i+1}$.
    \end{lemma}
\begin{proof}
Coefficient of $x^0$ is exactly first recurrence relation, coefficient of $x^r, 1 \leq r \leq k-1 $ is second and
coefficient of $x^k$ is third.
\end{proof}
Consider an operator $L: \bc[x] \to \bc[x], x^r \mapsto \sum_{i=1}^{r} \lambda_i x^{r-i}, 1 \mapsto 0$. Note that
$$L(xf) = xL(f)+l(f)$$ 
\begin{lemma}
\begin{eqnarray}(x+n-k)Q_k(x+1) = (x+n)Q_k(x) +L(Q_k)(x) \end{eqnarray}
\end{lemma}
\begin{proof}
We prove it by induction on $k$. For $k=0$ we have $x+n=x+n$.
Applying $L$ to (\ref{qrec}) we obtain
$$L(Q_k) = (x+n-k)L(Q_{k-1})+l(Q_{k-1})$$
Then
\begin{eqnarray*}
(x+n-k)Q_k(x+1) - (x+n)Q_k(x) - L(Q_k)(x) = \\ (x+n-k) \left[ (x+n-k)Q_{k-1}(x+1) - \frac{1}{k}l(Q_{k-1})\right] -(x+n)\left[(x+n-k)Q_{k-1}(x) - \frac{1}{k} l(Q_{k-1}) \right] \\ - (x+n-k)L(Q_{k-1})(x) - l(Q_{k-1}) = \\
(x+n-k)\left[ (x+n-k)Q_{k-1}(x+1) - (x+n)Q_{k-1}(x) - L(Q_{k-1}(x)  \right] = 0
\end{eqnarray*}
\end{proof}
For $k=n$ we have
$xQ_n(x+1) = (x+n)Q_n(x) + L(Q_n)(x)$. Note that $b_{n,n} = 1 = p_n$.
The coefficient of $x^r$ in this equality gives
$$b_{r-1,n} + \sum_{i=0}^{n-r} b_{r+i,n} \binom{r+i}{i+1}  =  b_{r-1,n}+n b_{r,n} + \sum_{i=1}^{n-r} b_{r+i,n} \lambda_i $$
hence
$$(r-n)b_{r,n} = \sum_{i=1}^{n-r} b_{r+i,n} \left(\lambda_i - \binom{r+i}{i+1} \right) = \sum_{t=r+1}^n b_{t,n}\left(\lambda_{t-r} - \binom{t}{r-1}\right) $$
This is exactly the recursion formula for $p_i$; hence $b_{i,n} = p_i$ and it completes the proof.
\end{proof}
Proposition \ref{prop: 3.10} implies that
$$\sum_{i=0}^{n}p_{i}f_{i}v=0.$$
and completes the proof.

\subsection{Proof of only if part}
Now suppose that $\lambda(u) = \dfrac{P(u+1)}{P(u)}$ for some monic polynomial $P(u)$ of degree $n$. 
If $P(u) = u-a$ then it is well-known that representation $\mathbb{C}^2(a) \simeq L_{\lambda(u)}.$ 

\begin{prop}
\label{prop: 3.12}
    Let $V$ and $W$ be finite-dimensional $Y(\fg)$-modules, and let $v \in V$, $w \in W$ be primitive vectors of weights $h_{V}$ and $h_{W}$ respectively. Then $v \otimes w$ is a primitive vector in $V \otimes W$ of weight $h_{V \otimes W}$, where 
    \[
    P_{V \otimes W}=P_{V}P_{W}
    \]
    and $P_{V \otimes W}$, $P_{V}$ and $P_{W}$ are the corresponding monic polynomials associated with the weights $h_{V \otimes W}$, $h_{V}$ and $h_{W}$.
\end{prop}
If $P(u) = (u-a_1) \ldots (u-a_n)$ then we consider representation
$$\mathbb{C}^2(a_1) \otimes \ldots \mathbb{C}^2(a_n)$$
which by proposition above contains a primitive vector of weight $\lambda(u) = \frac{P(u+1)}{P(u)}$. This implies that there exists a finite-dimensional highest-weight representation of weight $\lambda(u)$; hence $L_{\lambda(u)}$ is also finite-dimensional. 

\section{Representations of \texorpdfstring{$Y(\fg)$}{Y(g)}}

\subsection{If part}
Suppose that $\dim L_{\lambda(u)}<\infty$ and $v \in L_{\lambda(u)}$ be a highest weight vector. Then using Proposition \ref{prop: 2.10} we see that the module $Y_i v$ is finite-dimensional and classification for $Y(\fsl_2)$ case implies that $\lambda_i(u) = \frac{P(u+d_i)}{P_i(u)}$ for some unital $P_i(u) \in \bc[u]$. The appearance of $d_i$ can be explained as follows:
If for fixed $i$
    $$X_{i,r}^{+}v=0, \quad H_{i,r}v=\lambda_{i,r-1}v$$ for any $r \geq 0$ and  $H_{i,r}^{\prime} = \psi_i(H_{i,r}),  X_{i,r}^{+ \prime} = \psi_i(X_{i,r}^{+})$ then
    $$X_{i,r}^{+ \prime}v=0, \quad H_{i,r}^{\prime}v= d_i^{-r-1}\lambda_{i,r-1} v$$ for any $r \geq 0$.

\subsection{Only if part}

The proof of the only if part is the same as one for $U_q(\fg)$ in \cite{Pressley1995} and for the Yangian can be found in \cite{gautam}. We present a sketch of the proof for completeness.
Suppose that $\lambda_i(u)=\frac{P(u+d_{i})}{P(u)}$ for $i=1, \ldots, n$.
 As $\fg$-module we have decomposition
\[
L_{\lambda(u)}=\bigoplus_{\mu \in \mathfrak{h}^{*}}L_{\mu}, \quad \lambda(h_{i})=degP_{i} 
\]
Then we claim that:
\begin{itemize}
\item If $L_{\mu} \neq 0$ then  $\exists N_{\mu} > 0$ such that  $L_{\mu-k\alpha_{i}} = 0$ $\forall k \geq N$;

\item $dimL_{\mu} < \infty$.
\end{itemize}
As the Weyl group acts transitively on the set of weights and there are only finitely many dominant weights less than $\lambda$ from the statements above, it follows that $\dim L_{\lambda(u)} < \infty$. Let $v$ be a highest weight vector in $L_{\lambda(u)}$. The following is straightforward
\begin{lemma}
\label{lemma: 4.5}
    $V_i : = Y_{i}v$ is an irreducible module for all $i = \{1, \ldots, n \}$.
\end{lemma}
Note that classification for $Y(\fsl_2)$ implies that  $\dim V_i < \infty$.
The following lemma follows from commutation relations in the Yangian.
\begin{lemma}
\label{lemma: 4.6}
    $\forall r >0$ $L_{\mu-r\alpha_{i}}$ is spanned by vectors 
    \begin{equation}
x_{1}^{-}X_{i_{1},k_{1}}...x_{h}^{-}X_{i_{h},k_{h}}x_{h+1}^{-}v
\label{(5)}
    \end{equation}
    where $k_{1}, \cdots ,k_{h} \in \mathbb{N}$, are arbitrary and 
    \[
    x_{p}^{-}=X_{i,l_{1,p}}^{-}\cdots X_{i,l_{r_{p'}},p}^{-}, \quad l_{1,p}, \cdots , l_{r_{p'},p} \in \bz_{\geq 0}
    \]
    $r_{1}, ..., r_{h+1} \in \mathbb{N}$, and $r_{1}+...+r_{h+1}=r$, $r_{1},...,r_{h} \leq r'=max \{ -a_{ij} \}_{j \in I, j \neq i}$
\end{lemma}
Lemma \ref{lemma: 4.6} implies that  $L_{\mu-n\alpha_{i}}=0$ for $n>\lambda(h_{i}) + hr'$ proving first claim. For the second let $V_{i_{1},...,i_{m}} = \spann\left<X^-_{i_1,k_1} \ldots X^-_{i_m, k_m}  v  \right>$
Then again, from the commutation relations in the Yangian, it follows that
\[
V_{i_{1},...,i_{m}}=\sum_{k \geq 0} X_{i,k}^{-}V_{i_{2},...,i_{m}}
\]
and it proofs the second claim.
\subsection{Another proof of only if part.} From Proposition \ref{prop: 3.12}
it follows that in order to show that $$\dim L_{(\lambda_1(u), \ldots, \lambda_n(u))} < \infty$$ it is enough to show that {\it fundamental representations} of $Y(\fg)$ are finite-dimensional.

\begin{definition}
Representation $L_{\lambda(u)}$ is called fundamental if $\lambda_i(u) = 1$ for all $i \not = j$ and $\lambda_j(u) = \frac{u+d_j}{u}$, i.e. $P_i(u) = 1$ for $i \not = j$ and $P_j(u) = u$.
\end{definition}
For all types except $E_7, E_8$, it was done in \cite{Chari1991Fundamental}. For $E_7, E_8$, there is no complete proof in the literature; however, \cite{Kleber1997} gives a decomposition modulo the Kirillov-Reshetikhin conjecture, which was proved in \cite{DiFrancescoKedem2008} and \cite{Hernandez2006}. However, this works allows characters of fundamental $U_q(\hat \fg)$ modules and then using 
\cite{GautamToledanoLaredo2016} results on the equivalence of representation categories of $U_q(\hat \fg)$ and $Y(\fg)$, we obtain a result.

\begin{remark} 
In the paper \cite{Chari1991Fundamental}, in Section 2.16, the authors claim that if we have a certain fundamental representation of $Y(\fg)$ and know that it is finite-dimensional, then taking iterated tensor products one can show that any fundamental representation of $Y(\fg)$ is finite-dimensional. In fact, this is true for a particular fundamental representation and is not true for the $D_{2n}$ type. In the case of $U_q(\hat g)$, the statement was proved in \cite{etingof2003elliptic} by explicit computations. However, we cannot immediately get the same result for the Yangian, as the equivalence from \cite{GautamToledanoLaredo2016} generally does not preserve the tensor product.

\end{remark}

\section{Appendix A}
\begin{lemma}
\label{prop: 7}
    \[
    [h_{1}, e_{1}^{k}]=2ke_{1}^{k-1}e_{2}+k(3k-1)e_{1}^{k}+2ke_{1}^{k}h_{0}
    \]
\end{lemma}
\begin{proof}
We proceed by induction. Suppose that
\[
[h_{1}, e_{1}^{k}]=2ke_{1}^{k-1}e_{2}+k(3k-1)e_{1}^{k}+2ke_{1}^{k}h_{0}
\]
Now calculate 
\[
\begin{split}
    [h_{1}, e_{1}^{k+1}]=h_{1}e_{1}^{k}e_{1}-e_{1}^{k}e_{1}h_{1}=e_{1}^{k}h_{1}e_{1}+[h_{1}, e_{1}^{k}]e_{1}-e_{1}^{k}h_{1}e_{1}+e_{1}^{k}[h_{1}, e_{1}]=\\
    =[h_{1}, e_{1}^{k}]e_{1}+e_{1}^{k}[h_{1}, e_{1}]
\end{split}
\]
Using commutation relations, we obtain
\[
\begin{split}
    [h_{1}, e_{1}]-[h_{0}, e_{2}]=h_{0}e_{1}+e_{1}h_{0} \Leftrightarrow [h_{1}, e_{1}]-2e_{2}=2e_{1}h_{0}+2e_{1} \Leftrightarrow \\
    \Leftrightarrow [h_{1}, e_{1}]=2e_{2}+2e_{1}h_{0}+2e_{1}
\end{split}
\]
Then
\[
\begin{split}
    [h_{1}, e_{1}^{k+1}]=(2ke_{1}^{k-1}e_{2}+k(3k-1)e_{1}^{k}+2ke_{1}^{k}h_{0})e_{1}+e_{1}^{k}(2e_{2}+2e_{1}h_{0}+2e_{1})=\\
    =2ke_{1}^{k-1}e_{2}e_{1}+k(3k-1)e_{1}^{k}e_{1}+2ke_{1}^{k}h_{0}e_{1}+2e_{1}^{k}e_{2}+2e_{1}^{k+1}h_{0}+2e_{1}^{k+1}=\\
    =2ke_{1}^{k-1}e_{2}e_{1}+k(3k-1)e_{1}^{k+1}+2ke_{1}^{k}h_{0}e_{1}+2e_{1}^{k}e_{2}+2e_{1}^{k+1}h_{0}+2e_{1}^{k+1}
\end{split}
\]
Let us now use the following commutation relation
\[
[e_{2}, e_{1}]-[e_{1}, e_{2}]=e_{1}e_{1}+e_{1}e_{1} \Leftrightarrow [e_{2}, e_{1}]=e_{1}^{2}
\]
Thus
\[
2ke_{1}^{k-1}e_{2}e_{1}=2ke_{1}^{k}e_{2}+2ke_{1}^{k-1}[e_{2}, e_{1}]=2ke_{1}^{k}e_{2}+2ke_{1}^{k+1}
\]
Substituting this into our expression, we obtain
\[
\begin{split}
    [h_{1}, e_{1}^{k+1}]=2ke_{1}^{k}e_{2}+2ke_{1}^{k+1}+k(3k-1)e_{1}^{k+1}+2ke_{1}^{k}h_{0}e_{1}+2e_{1}^{k}e_{2}+2e_{1}^{k+1}h_{0}+2e_{1}^{k+1}
\end{split}
\]
Using the relation 
\[
[h_{1}, e_{1}]=2e_{1}
\]
we have 
\[
2ke_{1}^{k}h_{0}e_{1}=2ke_{1}^{k+1}h_{0}+4ke_{1}^{k+1}
\]
Thus 
\[
\begin{split}
    [h_{1}, e_{1}^{k+1}]=2ke_{1}^{k}e_{2}+2ke_{1}^{k+1}+k(3k-1)e_{1}^{k+1}+2ke_{1}^{k+1}h_{0}+4ke_{1}^{k+1}+2e_{1}^{k}e_{2}+2e_{1}^{k+1}h_{0}+2e_{1}^{k+1}=\\
    =2(k+1)e_{1}^{k}e_{2}+(2k+3k^{2}-k+4k+2)e_{1}^{k+1}+2(k+1)e_{1}^{k+1}h_{0}=\\
    =2(k+1)e_{1}^{k}e_{2}+(3k^{2}+5k+2)e_{1}^{k+1}+2(k+1)e_{1}^{k+1}h_{0}
\end{split}
\]
\[
3k^{2}+5k+2=3(k+1)(k+\frac{2}{3})=(k+1)(3k+2)=(k+1)(3(k+1)-1)
\]
Finally
\[
[h_{1}, e_{1}^{k+1}]=2(k+1)e_{1}^{k}e_{2}+(k+1)(3(k+1)-1)e_{1}^{k+1}+2(k+1)e_{1}^{k+1}h_{0}
\]
\end{proof}
We prove the following statement in a similar way.
\begin{lemma}
\label{prop: 8}
    \[
    [h_{1}, f_{0}^{k}]=-2kf_{0}^{k-1}f_{1}+k(3k-1)f_{0}^{k}-2kf_{0}^{k}h_{0}
    \]
\end{lemma}
\begin{proof}
    We proceed by induction. Suppose that
    \[
    [h_{1}, f_{0}^{k}]=-2kf_{0}^{k-1}f_{1}+k(3k-1)f_{0}^{k}-2kf_{0}^{k}h_{0}
    \]
    Now calculate 
    \[
    \begin{split}
        [h_{1}, f_{0}^{k+1}]=h_{1}f_{0}^{k}f_{0}-f_{0}^{k}f_{0}h_{1}=f_{0}^{k}h_{1}f_{0}+[h_{1}, f_{0}^{k}]f_{0}-f_{0}^{k}h_{1}f_{0}+f_{0}^{k}[h_{1}, f_{0}]=\\
        =[h_{1}, f_{0}^{k}]f_{0}+f_{0}^{k}[h_{1}, f_{0}]
    \end{split}
    \]
    Using commutation relations, we obtain
    \[
    \begin{split}
        [h_{1}, f_{0}]-[h_{0}, f_{1}]=-h_{0}f_{0}-f_{0}h_{0} \Leftrightarrow [h_{1}, f_{0}]+2f_{1}=-2f_{0}h_{0}+2f_{0} \Leftrightarrow \\
        \Leftrightarrow [h_{1}, f_{0}]=-2f_{1}-2f_{0}h_{0}+2f_{0}
    \end{split}
    \]
    Thus
    \[
    \begin{split}
        [h_{1}, f_{0}^{k+1}]=(-2kf_{0}^{k-1}f_{1}+k(3k-1)f_{0}^{k}-2kf_{0}^{k}h_{0})f_{0}+f_{0}^{k}(-2f_{1}-2f_{0}h_{0}+2f_{0})=\\
        =-2kf_{0}^{k-1}f_{1}f_{0}+k(3k-1)f_{0}^{k}f_{0}-2kf_{0}^{k}h_{0}f_{0}-2f_{0}^{k}f_{1}-2f_{0}^{k+1}h_{0}+2f_{0}^{k+1}=\\
        =-2kf_{0}^{k-1}f_{1}f_{0}+k(3k-1)f_{0}^{k+1}-2kf_{0}^{k}h_{0}f_{0}-2f_{0}^{k}f_{1}-2f_{0}^{k+1}h_{0}+2f_{0}^{k+1}
    \end{split}
    \]
    Commutation relation 
    \[
    [f_{1}, f_{0}]-[f_{0}, f_{1}]=-f_{0}f_{0}-f_{0}f_{0} \Leftrightarrow [f_{1}, f_{0}]=-f_{0}^{2}
    \]
    gives us
    \[
    -2kf_{0}^{k-1}f_{1}f_{0}=-2kf_{0}^{k}f_{1}-2kf_{0}^{k-1}[f_{1}, f_{0}]=-2kf_{0}^{k}f_{1}+2kf_{0}^{k+1}
    \]
    Substituting this into our expression, we obtain
    \[
    \begin{split}
        [h_{1}, f_{0}^{k+1}]=-2kf_{0}^{k}f_{1}+2kf_{0}^{k+1}+k(3k-1)f_{0}^{k+1}-2kf_{0}^{k}h_{0}f_{0}-2f_{0}^{k}f_{1}-2f_{0}^{k+1}h_{0}+2f_{0}^{k+1}=\\
        =-2(k+1)f_{0}^{k}f_{1}+2kf_{0}^{k+1}+k(3k-1)f_{0}^{k+1}-2kf_{0}^{k}h_{0}f_{0}-2f_{0}^{k+1}h_{0}+2f_{0}^{k+1}
    \end{split}
    \]
    Finally use
    \[
    [h_{0}, f_{0}]=-2f_{0}
    \]
    Then
    \[
    \begin{split}
        -2kf_{0}^{k}h_{0}f_{0}=-2kf_{0}^{k+1}h_{0}-2kf_{0}^{k}[h_{0}, f_{0}]=-2kf_{0}^{k+1}h_{0}+4kf_{0}^{k+1}
    \end{split}
    \]
    So we have
    \[
    \begin{split}
        [h_{1}, f_{0}^{k+1}]=-2(k+1)f_{0}^{k}f_{1}+2kf_{0}^{k+1}+k(3k-1)f_{0}^{k+1}-2kf_{0}^{k+1}h_{0}+4kf_{0}^{k+1}-2f_{0}^{k+1}h_{0}+2f_{0}^{k+1}=\\
        =-2(k+1)f_{0}^{k}f_{1}+(2k+3k^2-k+4k+2)f_{0}^{k+1}-2(k+1)f_{0}^{k+1}h_{0}=\\
        =-2(k+1)f_{0}^{k}f_{1}+(3k^{2}+5k+2)f_{0}^{k+1}-2(k+1)f_{0}^{k+1}h_{0}=\\
        =-2(k+1)f_{0}^{k}f_{1}+(k+1)(3(k+1)-1)f_{0}^{k+1}-2(k+1)f_{0}^{k+1}h_{0}
    \end{split}
    \]
\end{proof}
Using these two propositions, we can compute the commutator 
\[
\begin{split}
    [e_{1}^{k+1}, f_{0}]=-[f_{0}, e_{1}^{k+1}]=-\sum_{t=0}^{k}e_{1}^{t}[f_{0}, e_{1}]e_{1}^{k-t}=\sum_{t=0}^{k}e_{1}^{t}h_{1}e_{1}^{k-t}
\end{split}
\]
Using proposition \ref{prop: 7}, we can write
\[
\begin{split}
    [e_{1}^{k+1}, f_{0}]=\sum_{t=0}^{k}e_{1}^{t}(e_{1}^{k-t}h_{1}+[h_{1}, e_{1}^{k-t}])=\sum_{t=0}^{k}e_{1}^{k}h_{1}+\sum_{t=0}^{k}e_{1}^{t}[h_{1}, e_{1}^{k-t}]=\\
    =(k+1)e_{1}^{k}h_{1}+\sum_{t=0}^{k}e_{1}^{t}\Big( 2(k-t)e_{1}^{k-t-1}e_{2}+(k-t)(3k-3t-1)e_{1}^{k-t}+2(k-t)e_{1}^{k-t}h_{0} \Big)=\\
    =(k+1)e_{1}^{k}h_{1}+\sum_{t=0}^{k}\Big( 2(k-t)e_{1}^{k-1}e_{2}+(k-t)(3k-3t-1)e_{1}^{k}+2(k-t)e_{1}^{k}h_{0} \Big)=\\
    =(k+1)e_{1}^{k}h_{1}+2 \Big( \sum_{t=0}^{k} (k-t) \Big) e_{1}^{k-1}e_{2} + \Big( \sum_{t=0}^{k}(k-t)(3k-3t-1) \Big) e_{1}^{k} + 2\Big( \sum_{t=0}^{k} (k-t) \Big) e_{1}^{k}h_{0}
\end{split}
\]
Compute sums  
\[
\sum_{t=0}^{k}(k-t)=\frac{k(k+1)}{2} \quad \sum_{t=0}^{k}(k-t)(3k-3t-1)=k^{2}(k+1)
\]
Then write 
\[
[e_{1}^{k+1}, f_{0}]=(k+1)e_{1}^{k}h_{1}+k(k+1)e_{1}^{k-1}e_{2}+k^{2}(k+1)e_{1}^{k}+k(k+1)e_{1}^{k}h_{0}
\]
\begin{lemma}
\label{prop: 9}
\[
[e_{1}^{k+1}, f_{0}]=(k+1)e_{1}^{k}h_{1}+k(k+1)e_{1}^{k-1}e_{2}+k^{2}(k+1)e_{1}^{k}+k(k+1)e_{1}^{k}h_{0}
\]
\end{lemma}
In a similar way, we can compute the commutator:
\[
\begin{split}
    [e_{1}, f_{0}^{k+1}]=\sum_{t=0}^{k}f_{0}^{t}[e_{1}, f_{0}]f_{0}^{k-t}=\sum_{t=0}^{k}f_{0}^{t}h_{1}f_{0}^{k-t}
\end{split}
\]
Here we can use Proposition \ref{prop: 8}. We obtain
\[
\begin{split}
    [e_{1}, f_{0}^{k+1}]=\sum_{t=0}^{k}f_{0}^{t} \Big( f_{0}^{k-t}h_{1}+[h_{1}, f_{0}^{k-t}] \Big) =\sum_{t=0}^{k}f_{0}^{k}h_{1}+\sum_{t=0}^{k}f_{0}^{t}[h_{1}, f_{0}^{k-t}]=\\
    =(k+1)f_{0}^{k}h_{1}+\sum_{t=0}^{k}f_{0}^{t} \Big( -2(k-t)f_{0}^{k-t-1}f_{1}+(k-t)(3k-3t-1)f_{0}^{k-t}-2(k-t)f_{0}^{k-t}h_{0} \Big) =\\
    =(k+1)f_{0}^{k}h_{1}-2\Big( \sum_{t=0}^{k}(k-t) \Big) f_{0}^{k-1}f_{1}+\Big( \sum_{t=0}^{k}(k-t)(3k-3t-1) \Big) f_{0}^{k} - 2\Big( \sum_{t=0}^{k} (k-t) \Big) f_{0}^{k}h_{0}=\\
    =(k+1)f_{0}^{k}h_{1}-k(k+1)f_{0}^{k-1}f_{1}+k^{2}(k+1)f_{0}^{k}-k(k+1)f_{0}^{k}h_{0}
\end{split}
\]
\begin{lemma}
\label{prop: 10}
    \[
    [e_{1}, f_{0}^{k+1}]=(k+1)f_{0}^{k}h_{1}-k(k+1)f_{0}^{k-1}f_{1}+k^{2}(k+1)f_{0}^{k}-k(k+1)f_{0}^{k}h_{0}
    \]
\end{lemma}
Now we can carry out the main computation. 
\[
\begin{split}
    e_{1}^{k}f_{0}^{k+1}=e_{1}^{k}f_{0}f_{0}^{k}=f_{0}e_{1}^{k}f_{0}^{k}+[e_{1}^{k}, f_{0}]f_{0}^{k}
\end{split}
\]
Using proposition \ref{prop: 9}
\[
\begin{split}
    e_{1}^{k}f_{0}^{k+1}=f_{0}e_{1}^{k}f_{0}^{k}+\Big(ke_{1}^{k-1}h_{1}+k(k-1)e_{1}^{k-2}e_{2}+k(k-1)^{2}e_{1}^{k-1}+k(k-1)e_{1}^{k-1}h_{0} \Big) f_{0}^{k}=\\
    =f_{0}e_{1}^{k}f_{0}^{k}+ke_{1}^{k-1}h_{1}f_{0}^{k}+k(k-1)e_{1}^{k-2}e_{2}f_{0}^{k}+k(k-1)^{2}e_{1}^{k-1}f_{0}^{k}+k(k-1)e_{1}^{k-1}h_{0}f_{0}^{k}=\\
    =f_{0}e_{1}^{k}f_{0}^{k}+ke_{1}^{k-1}h_{1}f_{0}^{k}+k(k-1)\Big( e_{1}^{k-2}e_{2}+e_{1}^{k-1}h_{0} \Big) f_{0}^{k}+k(k-1)^{2}e_{1}^{k-1}f_{0}^{k}
\end{split}
\]
From the proposition \ref{prop: 7} 
\[
\begin{split}
    [h_{1}, e_{1}^{k-1}]=2(k-1)e_{1}^{k-2}e_{2}+(k-1)(3(k-1)-1)e_{1}^{k-1}+2(k-1)e_{1}^{k-1}h_{0} \Leftrightarrow \\
    \Leftrightarrow [h_{1}, e_{1}^{k-1}]=2(k-1)e_{1}^{k-2}e_{2}+(k-1)(3k-4)e_{1}^{k-1}+2(k-1)e_{1}^{k-1}h_{0} \Leftrightarrow \\
    \Leftrightarrow [h_{1}, e_{1}^{k-1}]=2(k-1) \Big( e_{1}^{k-2}e_{2}+e_{1}^{k-1}h_{0} \Big) +(k-1)(3k-4)e_{1}^{k-1} \Leftrightarrow \\
    \Leftrightarrow 2(k-1)(e_{1}^{k-2}e_{2}+e_{1}^{k-1}h_{0})=[h_{1}, e_{1}^{k-1}]-(k-1)(3k-4)e_{1}^{k-1} \Leftrightarrow \\
    \Leftrightarrow e_{1}^{k-2}e_{2}+e_{1}^{k-1}h_{0}=\frac{1}{2(k-1)}[h_{1}, e_{1}^{k-1}]-\frac{3k-4}{2}e_{1}^{k-1}
\end{split}
\]
Substituting, we obtain 
\[
\begin{split}
    e_{1}^{k}f_{0}^{k+1}=f_{0}e_{1}^{k}f_{0}^{k}+ke_{1}^{k-1}h_{1}f_{0}^{k}+k(k-1) \Big(\frac{1}{2(k-1)}[h_{1}, e_{1}^{k-1}]-\frac{3k-4}{2}e_{1}^{k-1} \Big) f_{0}^{k}+k(k-1)^{2}e_{1}^{k-1}f_{0}^{k}=\\
    =f_{0}e_{1}^{k}f_{0}^{k}+ke_{1}^{k-1}h_{1}f_{0}^{k}+\frac{k}{2}[h_{1}, e_{1}^{k-1}]f_{0}^{k}-\frac{k(k-1)(3k-4)}{2}e_{1}^{k-1}f_{0}^{k}+k(k-1)^{2}e_{1}^{k-1}f_{0}^{k}=\\
    =f_{0}e_{1}^{k}f_{0}^{k}+\frac{k}{2}[h_{1}, e_{1}^{k-1}]f_{0}^{k}+k(k-1) \Big( k-1-\frac{3k-4}{2} \Big)e_{1}^{k-1}f_{0}^{k}+ke_{1}^{k-1}h_{1}f_{0}^{k}=\\
    =f_{0}e_{1}^{k}f_{0}^{k}+\frac{k}{2}[h_{1}, e_{1}^{k-1}]f_{0}^{k}+k(k-1)\frac{-k+2}{2}e_{1}^{k-1}f_{0}^{k}+ke_{1}^{k-1}h_{1}f_{0}^{k}=\\
    =f_{0}e_{1}^{k}f_{0}^{k}+\frac{k}{2}[h_{1}, e_{1}^{k-1}]f_{0}^{k}-\frac{k(k-1)(k-2)}{2}e_{1}^{k-1}f_{0}^{k}+ke_{1}^{k-1}h_{1}f_{0}^{k}=\\
    =f_{0}e_{1}^{k}f_{0}^{k}+\frac{k}{2}[h_{1}, e_{1}^{k-1}]f_{0}^{k}+ke_{1}^{k-1}f_{0}^{k}h_{1}+ke_{1}^{k-1}[h_{1}, f_{0}^{k}]-\frac{k(k-1)(k-2)}{2}e_{1}^{k-1}f_{0}^{k}
\end{split}
\]
Thus, we have
\[
e_{1}^{k}f_{0}^{k+1}=f_{0}e_{1}^{k}f_{0}^{k}+ke_{1}^{k-1}f_{0}^{k}h_{1}+ke_{1}^{k-1}[h_{1}, f_{0}^{k}]+\frac{k}{2}[h_{1}, e_{1}^{k-1}]f_{0}^{k}-\frac{k(k-1)(k-2)}{2}e_{1}^{k-1}f_{0}^{k}
\]
Using proposition \ref{prop: 10}, we can write
\[
f_{0}^{k}h_{1}=\frac{1}{k+1}[e_{1}, f_{0}^{k+1}]+kf_{0}^{k-1}f_{1}-k^{2}f_{0}^{k}+kf_{0}^{k}h_{0}
\]
We obtain
\[
\begin{split}
    e_{1}^{k}f_{0}^{k+1}=f_{0}e_{1}^{k}f_{0}^{k}+ke_{1}^{k-1} \Big( \frac{1}{k+1}[e_{1}, f_{0}^{k+1}]+kf_{0}^{k-1}f_{1}-k^{2}f_{0}^{k}+kf_{0}^{k}h_{0} \Big)+\\
    +ke_{1}^{k-1}[h_{1}, f_{0}^{k}]+\frac{k}{2}[h_{1}, e_{1}^{k-1}]f_{0}^{k}-\frac{k(k-1)(k-2)}{2}e_{1}^{k-1}f_{0}^{k}
\end{split}
\]
Next, we consider the equality modulo the highest weight. Then
\[
[e_{1}, f_{0}^{k+1}]v=e_{1}f_{0}^{k+1}v
\]
So we have
\[
\begin{split}
    e_{1}^{k}f_{0}^{k+1}v=f_{0}e_{1}^{k}f_{0}^{k}v+\frac{k}{k+1}e_{1}^{k}f_{0}^{k+1}v+k^{2}e_{1}^{k-1}f_{0}^{k-1}f_{1}v-k^{3}e_{1}^{k-1}f_{0}^{k}v+k^{2}e_{1}^{k-1}f_{0}^{k}h_{0}v+\\
    +ke_{1}^{k-1}[h_{1}, f_{0}^{k}]v+\frac{k}{2}[h_{1}, e_{1}^{k-1}]f_{0}^{k}v-\frac{k(k-1)(k-2)}{2}e_{1}^{k-1}f_{0}^{k}v
\end{split}
\]
Collect on the left-hand side all terms containing $e_{1}^{k}f_{0}^{k+1}v$:
\[
\begin{split}
    \frac{1}{k+1}e_{1}^{k}f_{0}^{k+1}v=f_{0}e_{1}^{k}f_{0}^{k}v+k^{2}e_{1}^{k-1}f_{0}^{k-1}f_{1}v-k^{3}e_{1}^{k-1}f_{0}^{k}v+k^{2}e_{1}^{k-1}f_{0}^{k}h_{0}v+\\
    +ke_{1}^{k-1}[h_{1}, f_{0}^{k}]v+\frac{k}{2}[h_{1}, e_{1}^{k-1}]f_{0}^{k}v-\frac{k(k-1)(k-2)}{2}e_{1}^{k-1}f_{0}^{k}v
\end{split}
\]
The coefficient of $e_{1}^{k-1}f_{0}^{k}v$:
\[
\begin{split}
    -k^{3}-\frac{k(k-1)(k-2)}{2}=-\frac{2k^{3}+k(k^{2}-3k+2)}{2}=-\frac{3k^{3}-3k^{2}+2k}{2}=-\frac{k(3k^{2}-3k+2)}{2}
\end{split}
\]
Thus we can write
\[
\begin{split}
\frac{1}{k+1}e_{1}^{k}f_{0}^{k+1}v=f_{0}e_{1}^{k}f_{0}^{k}v+k^{2}e_{1}^{k-1}f_{0}^{k-1}f_{1}v+k^{2}e_{1}^{k-1}f_{0}^{k}h_{0}v+\\
+ke_{1}^{k-1}[h_{1}, f_{0}^{k}]v+\frac{k}{2}[h_{1}, e_{1}^{k-1}]f_{0}^{k}v-\frac{k(3k^{2}-3k+2)}{2}e_{1}^{k-1}f_{0}^{k}v=\\
=f_{0}e_{1}^{k}f_{0}^{k}v+k^{2}e_{1}^{k-1}(f_{0}^{k-1}f_{1}v+f_{0}^{k}h_{0}v)+ke_{1}^{k-1}[h_{1}, f_{0}^{k}]v+\\
+\frac{k}{2}[h_{1}, e_{1}^{k-1}]f_{0}^{k}v-\frac{k(3k^{2}-3k+2)}{2}e_{1}^{k-1}f_{0}^{k}v
\end{split}
\]
Using Lemma \ref{prop: 8}
\[
\begin{split}
    [h_{1}, f_{0}^{k}]=-2k(f_{0}^{k-1}f_{1}+f_{0}^{k}h_{0})+k(3k-1)f_{0}^{k} \Leftrightarrow \\
    \Leftrightarrow f_{0}^{k-1}f_{1}+f_{0}^{k}h_{0}=-\frac{1}{2k}[h_{1}, f_{0}^{k}]+\frac{3k-1}{2}f_{0}^{k}
\end{split}
\]
Substitute it into our relation:
\[
\begin{split}
\frac{1}{k+1}e_{1}^{k}f_{0}^{k+1}v=f_{0}e_{1}^{k}f_{0}^{k}v+k^{2}e_{1}^{k-1}(-\frac{1}{2k}[h_{1}, f_{0}^{k}]v+\frac{3k-1}{2}f_{0}^{k}v)+ke_{1}^{k-1}[h_{1}, f_{0}^{k}]v+\\
+\frac{k}{2}[h_{1}, e_{1}^{k-1}]f_{0}^{k}v-\frac{k(3k^{2}-3k+2)}{2}e_{1}^{k-1}f_{0}^{k}v=\\
=f_{0}e_{1}^{k}f_{0}^{k}v-\frac{k}{2}e_{1}^{k-1}[h_{1}, f_{0}^{k}]v+\frac{3k^{3}-k^{2}}{2}e_{1}^{k-1}f_{0}^{k}v+ke_{1}^{k-1}[h_{1}, f_{0}^{k}]v+\\
+\frac{k}{2}[h_{1}, e_{1}^{k-1}]f_{0}^{k}v-\frac{3k^{3}-3k^{2}+2k}{2}e_{1}^{k-1}f_{0}^{k}v=\\
=f_{0}e_{1}^{k}f_{0}^{k}v+\frac{k}{2}[h_{1}, e_{1}^{k-1}f_{0}^{k}]v+\frac{3k^{3}-k^{2}-3k^{3}+3k^{2}-2k}{2}e_{1}^{k-1}f_{0}^{k}v=\\
=f_{0}e_{1}^{k}f_{0}^{k}v+\frac{k}{2}[h_{1}, e_{1}^{k-1}f_{0}^{k}]v+\frac{2k^{2}-2k}{2}e_{1}^{k-1}f_{0}^{k}v
\end{split}
\]
Hence, we obtain the following expression:
\[
\begin{split}
\frac{1}{k+1}e_{1}^{k}f_{0}^{k+1}v=f_{0}e_{1}^{k}f_{0}^{k}v+\frac{k}{2}[h_{1}, e_{1}^{k-1}f_{0}^{k}]v+k(k-1)e_{1}^{k-1}f_{0}^{k}v \Leftrightarrow \\
\Leftrightarrow e_{1}^{k}f_{0}^{k+1}v=(k+1)f_{0}e_{1}^{k}f_{0}^{k}v+\frac{k(k+1)}{2}[h_{1}, e_{1}^{k-1}f_{0}^{k}]v+(k-1)k(k+1)e_{1}^{k-1}f_{0}^{k}v
\end{split}
\]

\bibliographystyle{alpha}
\bibliography{reference2}

\noindent\footnotesize{
{\bf Mikhail Gradov}  \\
HSE University, Moscow, Russia \\
{\tt gradov.m@hse.ru}\\

\noindent\footnotesize{
{\bf Aleksei Ilin}  \\
HSE University, Moscow, Russia \\
{\tt alex.omsk2@gmail.com}}\\

\end{document}

